\documentclass[12pt]{article}

\usepackage{a4wide,curves}
\usepackage{amsmath,amsthm,amsfonts,amssymb}
\usepackage{hyperref}
\usepackage{float}

\newtheorem{theorem}{Theorem}[section] 
\newtheorem{lemma}[theorem]{Lemma}%

\def\Aut{{\rm Aut}}
\def\PSp{{\rm PSp}}

\newcommand{\V}{\mathrm{V}}
\newcommand{\E}{\mathrm{E}}
\def\qed{\ifmmode\square\else\nolinebreak\hfill $\square$\fi\par\vskip12pt}

\theoremstyle{definition}

\date{}

\begin{document}
\baselineskip15pt 
\parskip 2pt

\title{Edge-transitive graphs of small order and the answer to a 1967 question by Folkman}

\author{
Marston Conder
\\[+2pt] 
{\normalsize Department of Mathematics, }\\[-3pt] 
{\normalsize University of Auckland,}\\[-3pt] 
{\normalsize Private Bag 92019, Auckland 1142, New Zealand} \\[+2pt] 
{\normalsize Email: m.conder@auckland.ac.nz} 
\\[+12pt] 
Gabriel Verret
\\[+2pt] 
{\normalsize Department of Mathematics, }\\[-3pt] 
{\normalsize University of Auckland,}\\[-3pt] 
{\normalsize Private Bag 92019, Auckland 1142, New Zealand} \\[+2pt] 
{\normalsize Email: g.verret@auckland.ac.nz} 
}

\maketitle

\begin{abstract}
In this paper, we introduce a method for finding all edge-transitive graphs of small order, 
using faithful representations of transitive permutation groups of small degree, 
and we explain how we used this method to find all edge-transitive graphs of order up to $47$, and 
all bipartite edge-transitive graphs of order up to $63$. 
We also give an answer to a 1967 question of Folkman about semi-symmetric 
graphs of large valency; in fact we show that for semi-symmetric graphs  of order $2n$ 
and valency $d$, the ratio $d/n$ can be arbitrarily close to $1$.
 ${}$ \\[+10pt]  
\end{abstract}

\section{Introduction} 
\label{sec:intro}

Graphs with a large automorphism group hold a significant place in mathematics, 
dating back to the time of first recognition of the Platonic solids, and also now in other 
disciplines where symmetry (and even other properties such as rigidity) play an important role, 
such as fullerene chemistry, and interconnection networks.

A major class of such graphs are the \emph{vertex-transitive} graphs, whose automorphism group 
has a single orbit on vertices.  Vertex-transitive graphs include many famous examples 
such as the Petersen graph and the Coxeter graph, and the underlying graph of the $C_{60}$ molecule, 
as well as infinite families including 
circulants, complete graphs, complete bipartite graphs $K_{n,n}$, generalised Petersen graphs, 
regular trees, and so on.  
Important sub-classes are those of \emph{Cayley graphs} (graphs for which  
some group of automorphisms acts sharply-transitively on vertices), and \emph{arc-transitive graphs}, 
which include the graphs underlying regular maps on surfaces, and the somewhat less well known 
class of \emph{half-arc-transitive graphs}, which are vertex- and edge-transitive but not arc-transitive. 


There is a close relationship between vertex-transitive graphs of given order $n$ and transitive permutation groups of degree $n$.  For if $X$ is a vertex-transitive graph of order $n$ with automorphism group $A$, and $\Delta$ is the neighbourhood of some vertex $v$, then $\Delta$ is preserved by $A_v$ and so $\Delta$ is a union of orbits of $A_v$. Conversely, let $G$ be any transitive group on a set $\Omega$ of size $n$.  Then for each union $\Delta$ of orbits of the stabiliser $G_v$ of a point $v \in \Omega$, one may define a graph on the set $\Omega$ with edges taken as the pairs of the form $\{v^g,w^g\}$ where $w \in \Delta$ and $g \in G$, and then $G$ is a subgroup of the automorphism group of $X$, acting transitively on vertices.  Hence all vertex-transitive graphs of $n$ can be found by constructing these  graphs for all possible choices of the pair $(G,\Delta)$, and checking for isomorphisms between them.

Using this method and the library of all transitive groups of degree at most $32$, Gordon Royle found all vertex-transitive graphs of order at most $32$ (see \cite{Royle-VTcensus31,Royle-VTcensus32}). Now that all transitive groups of degree  $33$ to $47$ are known (see \cite{Holt}), all vertex-transitive graphs of order up to $47$ can be found using the same method.  Currently this is a hard limit, in that the transitive groups of degree $48$ have not yet been determined. 

If the valency is  small, this process can be taken much further. 
All $3$-valent arc-transitive graphs of order up to 10000 
were found by the first author~\cite{C}, by exploiting the amalgams associated with such 
graphs \cite{DM80}, thereby considerably extending the `Foster census'. 
More recently, Primo{\v z} Poto\v cnik, Pablo Spiga and the second author determined all $3$-valent 
vertex-transitive graphs of order up to $1280$,  all 4-valent arc-transitive graphs of order up to $640$~\cite{PSV-cubicVT}, and  all $4$-valent half-arc-transitive graphs of order up to $1000$~\cite{PSV-HATval4}. 

In contrast, relatively little is known about  graphs that are edge-transitive but not vertex-transitive.  There is not even a good name for these graphs, and yet they include many infinite families of well-known examples including complete bipartite graphs $K_{m,n}$ with $m \ne n$ and, more generally, the incidence graphs of flag-transitive discrete structures 
(such as certain block designs and finite geometries).  If the graph is edge-transitive and also regular, 
but not vertex-transitive, then it is called \emph{semi-symmetric}, and a little more is known about these graphs, 
although mainly in the $3$-valent case (see \cite{Goldschmidt,CMMP} for example). 

In 2017, Brendan McKay asked us about finding all edge-transitive graphs  
of up to a certain order. In response, we developed a new method for solving this problem, which we describe in Section~\ref{sec:approach}.  Since a method already exists for vertex-transitive graphs, we focus on graphs that are edge-transitive but not vertex-transitive. Such graphs are necessarily bipartite. In fact, our new method finds all bipartite edge-transitive graphs up to a given order.  A key part of  it involves a reduction to the `worthy' case (where no two vertices have exactly the same neighbours), together with a `blow-up' construction to obtain all unworthy examples as well.  

Using this approach, we were able to determine all bipartite edge-transitive graphs of order up to $63$. Combining this with a list of vertex-transitive graphs yields a complete list of all edge-transitive graphs of order up to $47$. We describe the implementation of our methods and the results of computations in Sections~\ref{sec:implementation} and~\ref{sec:results}. 
The graphs themselves are listed on the websites indicated in references~\cite{CV1} and~\cite{CV2},  
and their edge-sets are given on websites subsidiary to those two.  (In fact, we found and give only connected graphs, but it is very easy to find the disconnected ones from these.)

While we began to write up and announce these results,  Heather Newman, Hector Miranda and Darren Narayan determined all edge-transitive graphs on up to $20$ vertices~\cite{NMN}.  We also discovered that some of our graphs answer one of the questions posed by Folkman in 1967 about semi-symmetric graphs of large valency~\cite{Folkman}. Inspired by these, in Section~\ref{sec:Folkman} we construct two infinite families of edge-transitive graphs which show that for semi-symmetric graphs of order $2n$ and valency $d$, the ratio $d/n$ can be arbitrarily close to~$1$.

\section{Further background} 
\label{sec:prelims}

In this paper, all graphs are assumed to be simple (with no loops or multiple edges) and undirected, and unless otherwise specified, also connected.  The vertex-set and edge-set of a graph $X$ will be denoted by $\V(X)$ and $\E(X)$,  respectively, and the neighbourhood of a vertex $v$ of $X$ by $X(v)$.  Also we use $\Aut(X)$ to denote the \emph{automorphism group} of $X$, namely the group of all permutations of the vertex-set that preserve adjacency. 

Next, we explain the `blow-up' construction, and define worthy and unworthy graphs.  

Let $Y$ be a bipartite graph, with parts $U$ and $W$, say, and let $(k,m)$ be a pair 
of positive integers.  The $(k,m)$-{\em blow-up} of $Y$ is obtained from $Y$ by replacing 
every vertex $u$ in $U$ by $k$ new vertices $u_1,u_2,\dots,u_k$ and every vertex $w$ in $W$ 
by $m$ new vertices $w_1,w_2,\dots,w_m$, and every edge $\{u,w\} \in U \times W$ by $km$ edges 
of the form $\{u_i,w_j\}$ for $1 \le i \le k$ and $1 \le j \le m$ (or in other words, by the edges of a complete 
bipartite subgraph isomorphic to $K_{k,m}$ between $\{u_1,u_2,\dots,u_k\}$ and $\{w_1,w_2,\dots,w_m\}$).   
For example, the graph $K_{k,m}$  is a $(k,m)$-blow-up of the $2$-vertex graph $K_{1,1}$, 
and is also a $(k,1)$-blow up of $K_{1,m}$, and a $(1,m)$-blow-up of $K_{k,1}$. 

Given a graph $X$, we can define an equivalence relation $\sim$ on its vertex-set by letting $v \sim w$ if and only if $v$ and $w$ have the same neighbourhood in $X$. Following~\cite{Wilson}, we say that $X$ is \emph{worthy} if $\sim$ is the identity relation (or in other words, if distinct vertices have distinct neighbourhoods).  Otherwise, $X$ is \emph{unworthy}.
Examples of unworthy graphs include the complete bipartite graphs $K_{m,n}$ with $m$ or $n$ at least $2$.  
Now let $X/{\hskip -3pt}\sim$ be the quotient graph obtained by collapsing every equivalence class to a single vertex. It is an easy exercise to show that $X/{\hskip -3pt}\sim$ is worthy, that $X/{\hskip -3pt}\sim$ is bipartite if only $X$ is bipartite, and that $X/{\hskip -3pt}\sim$ is edge-transitive if and only $X$ is edge-transitive. Finally, it is easy to see that $X$ is a blow-up of $X/{\hskip -3pt}\sim$. 

Combining all of this, we have the following.

\begin{lemma}
\label{lemma:blow-ups} 
Amongst edge-transitive bipartite connected graphs, every graph is the blow-up of a worthy one.
\end{lemma} 

\smallskip
We will also need the following easy lemmas.
\smallskip

\begin{lemma}
\label{worthy->faithful}
Let $X$ be a bipartite graph, and let $G$ be the part-preserving subgroup of $\Aut(X)$. 
If $X$ is worthy, then $G$ acts faithfully on each part of $X$.
\end{lemma}

\begin{proof}
Suppose to the contrary that $G$ does not act faithfully on one of the parts, say $U$.  Then some non-trivial element $g$ of $G$ fixes every vertex of $U$, but moves some vertex in the other part, say $w$.  But then since $X$ is bipartite and $g$ fixes $U$ pointwise, it follows that $w$ and $w^g$ have the same neighbourhood, contradicting the hypothesis that $X$ is worthy.
\end{proof}

\begin{lemma}
\label{lemma:ETnotVT}
A graph that is edge-transitive but not vertex-transitive is bipartite.
\end{lemma}

\begin{proof}
Let $X$ be such a graph, and let $A = \Aut(X)$.  If $\{v,w\}$ is any edge of $X$, then since $X$ is edge-transitive, every vertex lies in the $A$-orbit of one of $v$ and $w$, but not both (for otherwise $X$ would be vertex-transitive).  It follows that $X$ is bipartite, with its parts being those two $A$-orbits.  
\end{proof} 

\begin{lemma}
\label{lemma:automppties} 
Let $X$ be a edge-transitive bipartite connected graph, and let $G$ be the part-preserving subgroup of $\Aut(X)$. Then $G$ is transitive on each part of $X$, and either 
\\[+4pt] 
{\rm (a)}  $G_v$ is transitive on $X(v)$ for all $v \in \V(X)$, and on $\E(X)$, and $G=\langle G_v,G_w\rangle$ for all $\{v,w\} \in \E(X)$,  or
\\[+4pt] 
{\rm (b)}  $X$ is half-arc-transitive, $G_v$ has two orbits of equal size on $X(v)$, for all $v \in \V(X)$, and $G$ has two orbits on $\E(X)$, with representatives $\{v,w\}$ and $\{v,w'\}$ where $w$ and $w'$ are representatives for the two orbits of $G_v$ on $X_v$.
\end{lemma}
\begin{proof}
Let $A=\Aut(X)$, and let $H=\langle G_v,G_w\rangle$ where $\{v,w\}$ is an edge of $X$.  

Suppose first that $X$ is not vertex-transitive. Then since $X$ is edge-transitive, every vertex lies in the $A$-orbit of just one of $v$ and $w$, and it follows that $G = A$ and hence that $G$ is transitive on each part of $X$. Moreover, if $w_1,w_2\in X(v)$, then by edge-transitivity of $G$ ($=A$), there exists $g\in G$ taking $\{x,w_1\}$ to $\{x,w_2\}$, but $G$ preserves the bipartition of $X$, and so $g$ fixes $v$ and takes $w_1$ to $w_2$.  Hence $G_v$ is transitive on $X(v)$. Next, because $H$ contains $G_v$, it follows that all edges incident to $v$ are in the same orbit of $H$, and hence that $G_{w_1}$ is conjugate to $G_w$ by an element of $H$, so $G_{w_1} \leq H$. Repeating this argument inductively using connectedness shows that $H$ is transitive on $\E(X)$, and then since $H$ contains the edge-stabiliser $G_{\{v,w\}} = G_v \cap G_w$, we find that $H = G$.   Thus (a) holds. 

On the other hand, suppose that $X$ is vertex-transitive. Then $G$ is a subgroup of index $2$ in $A$, and is  transitive on each part of $X$. If also $X$ is arc-transitive, then $G$ is edge-transitive, and the arguments in the previous paragraph may be repeated to prove that (a) holds here as well.  Finally, if $X$ is not arc-transitive, then it is half-arc-transitive, and then $A_v$ has two orbits of the same size on $X(v)$, with the two arcs associated with every edge incident with $v$ lying in different orbits.  But also $A_v=G_v$,  and hence (b) holds.
\end{proof}

\section{Our approach} 
\label{sec:approach}

We now describe our approach to finding all small connected edge-transitive graphs of small order. 
We break this up into two cases, according to whether or not the graph is bipartite, and following that, we illustrate our approach for graphs of order $10$.

\medskip
\noindent 
{\bf Case (a)}: {\bf Non-bipartite edge-transitive graphs}  

In this case, all such graphs are vertex-transitive, by Lemma~\ref{lemma:ETnotVT}. 
We use the standard method described in the Introduction to find all small vertex-transitive graphs and then set aside those which are bipartite or not edge-transitive.  


\medskip
\noindent 
{\bf Case (b)}: {\bf Bipartite edge-transitive graphs}  

In this case, we consider only worthy graphs, since every unworthy example can be constructed as a blow-up of a worthy example, by Lemma~\ref{lemma:blow-ups}. 

Now let $X$ be a worthy bipartite edge-transitive graph, with parts $U$ and $W$, and let $G$ be the subgroup of $\Aut(X)$ preserving the parts $U$ and $W$.  Note that $G$ might not act transitively on $\E(X)$, but we know that $U$ and $W$ are the orbits of $G$ on $\V(X)$ by Lemma~\ref{lemma:automppties}, and that $G$ acts faithfully on each of $U$ and $W$ by Lemma~\ref{worthy->faithful}.  It follows that we can think of $G$ as a transitive permutation group on $U$, with an auxiliary transitive action on $W$.  Also if $u\in U$, then by Lemma~\ref{lemma:automppties} we know that either (a) $G_u$ is transitive on $X(u)$, and $G$ is transitive on $\E(X)$, or (b) $G_u$ has two orbits of the same size on $X(u)$, and $G$ has two orbits on $\E(X)$.

These observations lead us to the following algorithm, for finding all {\em worthy} edge-transitive bipartite connected graphs with parts of sizes $k$ and $m$.

\medskip\smallskip 
\noindent 
{\bf Algorithm:} \\[-16pt]  
\begin{enumerate}
\item Find all transitive permutation groups of degree $k$. \\[-18pt] 
\item For each such group $G$ acting transitively on a set $U$ of size $k$, find all faithful transitive permutation representations of $G$ on a set $W$ of size $m$.  \\[-18pt] 
\item For each such representation of $G$ on $W,$ choose a point $u\in U$, and then \\[-20pt] 
\begin{enumerate}
\item for each orbit ${\cal O}$ of $G_u$ on  $W,$ choose a point $w\in{\cal O}$ and then construct the graph $X$ with vertex-set $U \cup W$ and with edge-set the orbit of  $\{u,w\}$ under $G$ (in its natural induced action on $U \cup W$),
\item  for each pair of equal-sized orbits ${\cal O}$ and ${\cal O}'$ of $G_u$ on $W$, choose points $w\in{\cal O}$ and $w'\in{\cal O'}$ and then construct the graph $Y$ with vertex-set $U \cup W$ and with edge-set the union of the orbits of $\{u,w\}$ and $\{u,w'\}$ under $G$.  \\[-18pt] 
\end{enumerate}
\item Check each graph obtained for connectedness, worthiness and isomorphism with previously found graphs, and also the graphs found in 3(b) for edge-transitivity.  
\end{enumerate}

Finally, to find all connected edge-transitive bipartite graphs with parts of sizes $k$ and~$m$, we construct the $(k/k',m/m')$-blow-ups of all worthy edge-transitive bipartite connected graphs with  parts of sizes $k'$ and $m'$, for each divisor $k'$ of $k$ and each divisor $m'$ of~$m$.  

\bigskip
\noindent\textbf{Worked example: Edge-transitive connected graphs of order $10$}
\medskip

To illustrate our method, we explain how it can be applied to find all edge-transitive graphs of order $10$.  

In case (a), there are $45$ transitive groups of order $10$, and by the standard method, these give rise to $22$ different vertex-transitive graphs of order $10$, given at~\cite{Royle-VTcensus31}, with $18$ being connected.  Of these, just eight are edge-transitive, 
and five are non-bipartite, namely the Petersen graph and its complement, the circulant graph $C(10,\{2,3\})$, the complete graph $K_{10}$, and the graph $K_{10}-5K_2$ obtained from it by removing the edges of a perfect matching.  

In case (b), by swapping the parts $U$ and $W$ of the graph if necessary, we may assume that $k = |U|\leq |W| = m$, and hence the possibilities for the pair $(k,m)$ are $(1,9)$, $(2,8)$, $(3,7)$, $(4,6)$ and $(5,5)$. 
We will also assume that we know the worthy edge-transitive bipartite connected graphs of order less than $10$, listed at~\cite{CV2}.  The unworthy graphs of order $10$ constructible as blow-ups of these are $K_{1,9}$, $K_{2,8}$, $K_{3,7}$, $K_{4,6}$ and $K_{5,5}$.  From this point on we will seek only worthy edge-transitive bipartite examples, using the algorithm given above.  
When  $(k,m) = (1,9)$, $(2,8)$ or $(3,7)$, there is no such graph, because no transitive group of degree at most $3$ has a transitive representation of degree greater than~$6$. 
Hence we need only consider  the cases where $(k,m) = (4,6)$ or $(5,5)$.  

In seeking worthy graphs when $(k,m) = (4,6)$, we take $G$ as a transitive permutation group of degree $4$ 
with a transitive permutation representation of degree $6$.  Clearly the only possibilities are $A_4$ and $S_4$, 
and we must take their natural actions on $U$, and their actions on $W$ as the standard permutation 
representations on cosets of a subgroup of order $2$ (for $A_4$) or of order $4$ (for $S_4$). 

When $G = A_4$, the representation of $G$ on $W$ is unique (up to equivalence), 
and the stabiliser of a point in $U$ has two orbits of length $3$ on $W$.  Taking either one of these orbits 
give a worthy (but not vertex-transitive) connected graph $X$ with $12$ edges, such that the vertices 
in $U$ and $W$ have valencies $3$ and $2$, respectively.  The two graphs obtained in this way are isomorphic, and are 
given by ET10.2 and ETB10.2 in the lists in~\cite{CV1} and~\cite{CV2}.   Taking both orbits gives the graph $K_{4,6}$, which is unworthy and so can be ignored. 

On the other hand, when $G = S_4$ there are three possibilities for the representation of $G$ on $W$, 
because $S_4$ has three conjugacy classes of subgroups of order $4$.  
The representation on cosets of the subgroup $V_4$ is unfaithful (and gives the unworthy graph $K_{4,6}$), so can be ignored.  Similarly, for the representation on cosets of a cyclic subgroup of order $4$ can be discarded, because the stabiliser of a point in $U$ has a single orbit of length $6$ on $W$, and again gives the unworthy graph $K_{4,6}$. 
For the representation on cosets of the subgroup generated by $(1,2)$ and $(1,2)(3,4)$, again there 
are two orbits of length $3$ on $W$, and for both of them, the same graph arises as the one for $G = A_4$. 
Hence there is just one worthy edge-transitive bipartite connected graph with parts of sizes $4$ and $6$. 

Next, in the case where $(k,m) = (5,5)$, the permutation group $G$ can be $C_5$, $D_5$, 
$C_5 \rtimes_2 C_4$ (as a Frobenius group), $A_5$ or $S_5$. 
The first possibility  $G = C_5$ gives only the disconnected graph $5K_2$. 
For the second, where $G = D_5$, up to equivalence there is just one transitive action 
of degree $m = 5$, and for that, the stabiliser of a point in $U$ has three orbits on $W$, 
of lengths $1$, $2$ and $2$.  The first orbit gives $5K_2$ again, while each of the second and third gives the 
cycle graph $C_{10}$, which is connected, worthy and vertex-transitive (indeed arc-transitive), 
and the union of those two orbits gives the graph $K_{5,5}-5K_2$ (obtainable from $K_{5,5}$ by removing the edges of a perfect matching), which is also connected, worthy and arc-transitive.  
When $G$ is the Frobenius group $C_5 \rtimes_2 C_4$, again there is just one transitive action 
of degree $m = 5$, and for this representation, the stabiliser of a point in $U$ has two orbits on $W$, 
of lengths $1$ and $4$, and the first gives $5K_2$ again, while the second gives $K_{5,5}-5K_2$ again. 
The same this happens also for $G = A_5$ and $S_5$.
Hence there are just two worthy edge-transitive bipartite connected graphs with both parts of sizes $5$.   

In summary, there are $13$ edge-transitive graphs of order $10$, with five being vertex-transitive and non-bipartite, 
five being bipartite and unworthy, and three being  bipartite and worthy.  Also none of these graphs is half-arc-transitive.

\section{Implementation}\label{sec:implementation}
We implemented our approach from the previous section using the {\sc Magma} system \cite{Magma}.  This allowed us to determine all connected non-bipartite edge-transitive graphs of order up to $47$. Once again, order $47$ is a hard limit at this point in time, because the transitive groups of degree $48$ have not yet been determined. 

The other matters about this implementation concern only the bipartite case, and we give relevant details for each step of the algorithm below:
\\[-22pt]
\begin{enumerate}
\item We used the library of transitive groups of degree at most $47$ in {\sc Magma}; the most recent part of this comes from~\cite{Holt}.
\\[-18pt]
\item  To find all faithful transitive permutation representations of the group $G$ on the part of size $m$, we find  all core-free subgroups of index $m$ in $G$. In some cases, this is very computationally intensive, and the following observation can be very helpful:
\\[+4pt] 
Suppose that $G$ has a normal Hall subgroup $N$ (that is, with order $|N|$ coprime to its index $|G/N|$).  
If $H$ is a subgroup of index $m$ in $G$, then the image of $H$ in $G/N$ is a subgroup 
of index $d = \gcd(m,|G\!:\!N|)$, 
as is its pre-image $J$ in $G$, and $H$ is then a subgroup of index $m/d$ in $J$. 
This means we can find all possibilities for $H$ by first looking for index $d$  
subgroups in $G$, and then for index $m/d$ subgroups of those.  
\\[-18pt]
\item  In step 3(a) we must have $\langle G_u,G_w\rangle = G$, by part (a) of Lemma~\ref{lemma:automppties}. This condition can be checked early and quickly, before constructing the graph $X$.
\\[-18pt]
\item Note that a graph that results from step 3 might not be connected or worthy, so these properties have to be checked at the end. On the other hand, it will certainly be bipartite and, if constructed in (3a), edge-transitive.
\end{enumerate}

With this approach, the computations to find all connected worthy bipartite edge-transitive graphs of order up to $23$ took only minutes, and less than three hours for those of order up to $47$. 

For larger graphs, the computations were much longer, owing to the number of groups to consider: there are  $25000$, $2801324$, $121279$ and $315842$ 
transitive permutation groups of degrees $24$, $32$, $36$ and $40$, respectively. In many cases, we can simply swap the role of $k$ and $m$, without loss of generality. Together with the tricks above, this allowed us to deal with all cases with $k+m\leq 63$, the bottleneck to further progress being the case $(k,m) = (32,32)$. This case seems out of reach of our methods at the moment.

In the end, we found all connected bipartite edge-transitive graphs of order at most $63$; see~\cite{CV2}. When combined with the non-bipartite ones, these gave us all connected edge-transitive graphs of order at most $47$; see~\cite{CV1}. To go further than this with our approach, one would need to know the transitive permutation groups of degree $48$. 

\section{Summary of our results} 
\label{sec:results}


There are $1894$ non-isomorphic edge-transitive  connected graphs of order up to $47$.  
Of these, $1429$ are bipartite while $465$ are non-bipartite, 
and $625$ are worthy while $1269$ are unworthy, and 
$678$ are vertex-transitive while $1216$ are not, and of the $678$ vertex-transitive graphs, 
$670$ are arc-transitive while $8$ are half-arc-transitive. 
Similarly, there are $3309$ bipartite edge-transitive connected graphs of order up to $63$, 
of which $792$ are worthy while $2517$ are unworthy, 
and $435$ are vertex-transitive while $2874$ are not.   

A more detailed breakdown  is 
given in Table~\ref{ETgraphdatatable} below, with `Tot' indicating the total number of such 
graphs of order $n$, and then `Reg', `Bpte', `VT', `AT' and `Wthy' indicating the number of those 
that are regular, bipartite, vertex-transitive, arc-transitive and worthy, respectively. 

Other information can be obtained directly from the lists at~\cite{CV1} and~\cite{CV2}, 
or from the first author on request.

\begin{table}[ht] 
\label{ETgraphdatatable} 
\caption{Summary data for connected edge-transitive graphs of  order $n\leq 47$} 
\noindent
\begin{center}
\begin{tabular}{||c|c|c|c|c|c|c||c|c|c|c|c|c|c||}
\hline\hline
$n$ & Tot & Reg & Bpte & VT & AT & Wthy & $n$ & Tot & Reg & Bpte & VT & AT & Wthy \\
\hline\hline
1 & 1 & 1 & 0 & 1 & 1 & 1 & 25 & 34 & 11 & 23 & 11 & 11 & 12 \\
\hline 
2 & 1 & 1 & 1 & 1 & 1 & 1 & 26 & 31 & 13 & 26 & 13 & 13 & 10 \\
\hline 
3 & 2 & 1 & 1 & 1 & 1 & 1 & 27 & 51 & 21 & 30 & 21 & 20 & 21 \\
\hline 
4 & 3 & 2 & 2 & 2 & 2 & 1 & 28 & 64 & 27 & 47 & 26 & 26 & 25 \\
\hline 
5 & 4 & 2 & 2 & 2 & 2 & 2 & 29 & 18 & 4 & 14 & 4 & 4 & 4 \\
\hline 
6 & 6 & 4 & 4 & 4 & 4 & 2 & 30 & 93 & 41 & 66 & 41 & 41 & 30 \\
\hline 
7 & 5 & 2 & 3 & 2 & 2 & 2 & 31 & 19 & 4 & 15 & 4 & 4 & 4 \\
\hline 
8 & 8 & 5 & 6 & 5 & 5 & 3 & 32 & 83 & 45 & 65 & 42 & 42 & 32 \\
\hline 
9 & 9 & 4 & 5 & 4 & 4 & 3 & 33 & 44 & 8 & 36 & 8 & 8 & 8 \\
\hline 
10 & 13 & 8 & 8 & 8 & 8 & 6 & 34 & 34 & 10 & 29 & 10 & 10 & 7 \\
\hline 
11 & 7 & 2 & 5 & 2 & 2 & 2 & 35 & 67 & 15 & 52 & 15 & 15 & 19 \\
\hline 
12 & 19 & 11 & 12 & 11 & 11 & 6 & 36 & 154 & 75 & 107 & 69 & 67 & 67 \\
\hline 
13 & 10 & 4 & 6 & 4 & 4 & 4 & 37 & 24 & 6 & 18 & 6 & 6 & 6 \\
\hline 
14 & 16 & 8 & 13 & 8 & 8 & 6 & 38 & 36 & 10 & 32 & 10 & 10 & 6 \\
\hline 
15 & 25 & 10 & 15 & 10 & 10 & 11 & 39 & 60 & 14 & 46 & 14 & 12 & 14 \\
\hline 
16 & 26 & 15 & 18 & 15 & 15 & 11 & 40 & 175 & 79 & 132 & 71 & 68 & 71 \\
\hline 
17 & 12 & 4 & 8 & 4 & 4 & 4 & 41 & 26 & 6 & 20 & 6 & 6 & 6 \\
\hline 
18 & 28 & 14 & 21 & 14 & 14 & 8 & 42 & 147 & 58 & 114 & 56 & 56 & 55 \\
\hline 
19 & 12 & 3 & 9 & 3 & 3 & 3 & 43 & 25 & 4 & 21 & 4 & 4 & 4 \\
\hline 
20 & 43 & 24 & 29 & 22 & 22 & 15 & 44 & 88 & 17 & 80 & 16 & 16 & 17 \\
\hline 
21 & 37 & 13 & 24 & 13 & 13 & 15 & 45 & 161 & 42 & 119 & 42 & 42 & 59 \\
\hline 
22 & 24 & 8 & 21 & 8 & 8 & 7 & 46 & 46 & 7 & 43 & 7 & 7 & 7 \\
\hline 
23 & 13 & 2 & 11 & 2 & 2 & 2 & 47 & 25 & 2 & 23 & 2 & 2 & 2 \\
\hline 
24 & 65 & 36 & 47 & 34 & 34 & 23 & All & 1894 & 703 & 1429 & 678 & 670 & 625 \\
\hline 
\hline
\end{tabular}
\end{center}
\end{table}

\section{Answers to Folkman's questions} 
\label{sec:Folkman}

In a seminal paper \cite{Folkman}, Folkman investigated regular graphs that are edge-transitive 
but not vertex-transitive, and asked eight questions at the end about such graphs, 
which he called `admissible' (but are now known as semi-symmetric).  

Two of these questions (4.1 and 4.8 in \cite[Section 4]{Folkman}) were general ones about 
the orders and valencies of semi-symmetric graphs, and remain open (and might never be answered). 

Another three of them (4.5 to 4.7) were about the existence of semi-symmetric graphs 
of order $2n$ and valency $d$ where $d$ is prime, or $d$ is coprime to $n$, or $d$ is a prime 
that does not divide $n$, and these have been answered by the construction of various 
semi-symmetric 3-valent graphs (including examples of orders 110 and 112).  
Question 4.2 asked if there is a semi-symmetric graph of order 30, 
and this was answered in 1987 by Ivanov, who proved in \cite{Ivanov} that no such graph exists, 
and the results of our computations confirm this. 
Question 4.3 asked if there a semi-symmetric graph of order $2pq$, where $p$ and $q$ 
are odd primes such that $p < q$, and $p$ does not divide $q-1$, and this was answered in 2000 by Du and Xu, 
who determined in \cite{DuXu} all semi-symmetric graphs of order $2pq$ where $p$ and $q$ 
are distinct primes; these included graphs of orders $70$ ($= 2 \cdot 3 \cdot 5$), 
$154$ ($= 2 \cdot 7 \cdot 11$)  and $3782$ ($= 2 \cdot 31 \cdot 61$).  
 
The remaining question (4.4) asked if there exists a semi-symmetric graph of order $2n$ 
and valency $d$ where $d \ge n/2$.  As far as we know, this question has remained unanswered since 1967, 
because until now, the valency of most known semi-symmetric graphs is relatively small.  
(In contrast, the valency of symmetric bipartite graphs can be relatively large; 
indeed the complete bipartite graph $K_{n,n}$ is symmetric and has valency $n$, for all $n$.) 

Our computations answer Folkman's question positively.  
For example, the graphs ET20.15, ET24.20 and  ET36.80 
have the required property, with valencies $6$ ($= 3n/5$), $6$ ($= n/2$) and $12$ ($= 2n/3$) respectively. 
In fact all of these graphs are blow-ups of smaller non-regular edge-transitive graphs, and they provide the idea 
behind the construction in the proof of the following theorem, which takes Folkman's question much further 
by showing that the ratio $d/n$ can be arbitrarily close to $1$.  

\begin{theorem}
\label{thm:Folkman4.4} 
For every integer $k \ge 3$, there exists a semi-symmetric graph of order $4k^2$ 
with valency $2k(k\!-\!1)$, giving the valency to part-size ratio $d/n$ as $(k\!-\!1)/k$.  
Hence in particular, the valency to part-size ratio tends to $1$ as $k \to \infty$. 
\end{theorem}

\begin{proof}
For every integer $k \ge 3$, let $A$ be the union of two disjoint sets $A_1$ and $A_2$ of size $k$, 
and let $B = A_1 \times A_2$, and make $A$ and $B$ the parts of a bipartite graph $X$ 
in which the edges join each vertex $(a_1,a_2) \in B$ to every vertex $a \in A_1\setminus \{a_1\}$
and to every vertex $a \in A_2 \setminus \{a_2\}$.  Then $X$ has $2k\!+\!k^2$ vertices and $2k^2(k\!-\!1)$ edges, with each vertex $a \in A$ having valency $k(k\!-\!1)$ 
and each vertex $b \in B$ having valency $2(k\!-\!1)$. Also $X$ admits a natural action of the wreath product $S_k \wr C_2$ as a group of automorphism (and in fact $\Aut(X)=S_k \wr C_2$, but we will not need this), so $X$ is edge-transitive.  

Now take the $(k,2)$ blow-up of $X$.  The result is an edge-transitive regular bipartite graph $Y$ 
of order $2n = 4k^2$ and valency $d = 2k(k\!-\!1)$.  But $Y$ is not vertex-transitive, since every vertex 
in the blow-up of $B$ has the same neighbourhood as just one other vertex, while every vertex in the blow-up 
of $A$ has the same neighbourhood as exactly $k\!-\!1$ other vertices. 
Hence the graph $Y$ is semi-symmetric, as claimed.   The rest follows easily. 
\end{proof}

The smallest example in this family is the graph ET36.80 (mentioned earlier).  

It can be observed, however, that the semi-symmetric graphs in Construction 1 are unworthy, 
as indeed are all the other examples of order less than 63. 
We are grateful to Primo{\v z} Poto{\v c}nik, who asked us if Folkman's 
question can also be answered positively for semi-symmetric graphs that are worthy, 
and now show that to be so.  We are also grateful to Michael Giudici, who helped us 
identify the first small example that we found, which led to the construction used in the following. 

\begin{theorem}
\label{thm:Folkman4.4q} 
For every odd prime-power $q$, there exists a worthy semi-symmetric graph of order $2(q^{3}+q^{2}+q+1)$ 
with valency $q^{3}+q^{2}$, giving the valency to part-size ratio $d/n$ as $(q^{3}+q^{2})/(q^{3}+q^{2}+q+1)$.  
In particular, this ratio tends to $1$ as $q \to \infty$. 
\end{theorem}

\begin{proof}
For every such $q$, let $Q$ be the generalised quadrangle 
associated with a symplectic form on $V = {\mathbb F}_q^4$ 
(such as $\,\langle {\bf x}, {\bf y} \rangle = x_1y_3 + x_2y_4 - x_3y_1 - x_4y_2).\,$    
Then $Q$ has $q^{3}+q^{2}+q+1$ points (which are the 1-dimensional subspaces of~$V$) 
and $q^{3}+q^{2}+q+1$ totally isotropic lines  (which are the 2-dimensional subspaces $U$ of $V$ 
for which $U = U^{\perp}$), and point-line incidence is given by natural inclusion.   
In fact every point lies on $q+1$  totally isotropic lines, and every totally isotropic line contains $q+1$ points. 
This geometry is flag-transitive, but not self-dual (by a theorem of Benson \cite{B}), 
and hence the Levi (incidence) graph of $Q$ is regular with valency $q+1$, and is edge-transitive (indeed locally arc-transitive), but not vertex-transitive. 
Moreover, its automorphism group is isomorphic to $\Aut(\PSp(4,q))$, 
and acts primitively on both parts of the graph.  

Now take the bipartite complement of this Levi graph, in which each point is joined to each of the 
totally isotropic lines that do {\em not\/} contain it.  
The resulting graph is regular with valency $q^3+q^2$, and is edge-transitive but not vertex-transitive, 
and its automorphism group is also isomorphic to $\Aut(\PSp(4,q))$. 
In particular, this graph is semi-symmetric, but also worthy, since its automorphism group 
acts primitively on each of the two parts of the graph. 
Finally, for this graph the ratio $d/n$ is $(q^{3}+q^{2})/(q^{3}+q^{2}+q+1)$, which tends to $1$ as $q \to \infty$. 
\end{proof}

\medskip
\bigskip

\noindent
{\Large\bf Acknowledgements}
\medskip

The authors acknowledge the extensive use of the {\sc Magma} system \cite{Magma} 
in testing and carrying out the computations described in this paper.  
Also the first author is grateful to the N.Z.\ Marsden Fund for its support 
(via grant UOA1626). 



\end{document}